\documentclass[12pt]{article}   
\usepackage{graphicx,psfrag}
\usepackage[table]{xcolor}
\usepackage{color}
\usepackage{tikz}
\usepackage{array}
\usepackage{colortbl}
\usepackage{amsmath,amsfonts}
\usepackage{amssymb,amsthm}

\usepackage{epstopdf}

\usepackage{tabularx}
\newtheorem{theorem}{Theorem}[section]
\newtheorem{corollary}[theorem]{Corollary}
\newtheorem{proposition}[theorem]{Proposition}

\newtheorem{example}[theorem]{Example}
\newtheorem{lemma}[theorem]{Lemma}

\def\qed{\vbox{\hrule
 \hbox{\vrule\hbox to 5pt{\vbox to 8pt{\vfil}\hfil}\vrule}\hrule}}

\def\endproof{\unskip \nobreak \hskip0pt plus 1fill \qquad \qed \par}

\newcommand{\mc}{\mathcal}

\usepackage{tikz}
\usetikzlibrary{decorations.markings}

\begin{document}

\title{Convex $(0,1)$-Matrices and Their Epitopes}

\author{
 Richard A. Brualdi\footnote{Department of Mathematics, University of Wisconsin, Madison, WI 53706, USA. {\tt brualdi@math.wisc.edu}} \\
 Geir Dahl\footnote{Department of Mathematics,  
 University of Oslo, Norway.
 {\tt geird@math.uio.no.} Corresponding author.}
 }
\date{12 August 2020}
\maketitle

\begin{abstract}  We investigate $(0,1)$-matrices that are {\em convex}, which means that the ones are consecutive in every row and column. These matrices occur in discrete tomography. The notion of ranked essential sets, known for permutation matrices, is extended to convex sets. We show a number of results for  the class $\mc{C}(R,S)$ of convex matrices with given row and column sum vectors $R$ and $S$.  Also, it is shown that the ranked essential set uniquely determines a matrix in $\mc{C}(R,S)$. 
 \end{abstract}




\section{Introduction}

Mathematical reconstruction problems arise in the area of image analysis, and discrete tomography (\cite{Alpers17, Balazs07, RABDD1,DF}) is such a field where combinatorics plays a major role. A $(0,1)$-matrix then represents an image with black-white pixels, and the row and column sums count the number of black pixels in that line. Within this framework convex matrices, i.e., those with consecutive ones, are interesting. In this paper we study several questions for such convex matrices. In particular, we study so-called ranked essential sets for convex matrices.

A $(0,1)$-matrix $A$ is {\it row convex} provided the 1's in each row occur consecutively;  {\it column convex} is defined  analogously. The matrix $A$ is {\it convex} provided it is both row and column convex.
Permutation matrices are trivially convex. Let  $R=(r_1,r_2,\ldots,r_m)$ and $S=(s_1,s_2,\ldots,s_n)$ be nonnegative integral vectors with $\sum_{i=1}^m r_i=\sum_{j=1}^n s_j$. Let ${\mathcal A}(R,S)$ be the set of all $m\times n$ (0,1)-matrices with row sum vector
$R$ and column sum vector $S$.  The class of $m\times n$ convex $(0,1)$-matrices in ${\mathcal A}(R,S)$ is denoted by ${\mathcal C}_{m,n}(R,S)$ or, more simply as  ${\mathcal C}(R,S)$. It is an {\em NP}-complete problem to determine whether a class ${\mathcal C}_{m,n}(R,S)$ is nonempty \cite{DF,W}.

\begin{example}\label{ex:convex0} {\rm Below are convex matrices with $R=(4,1,1,2,1,1)$ and $S=(2,1,3,1,$ $1,1,1)$, and $R=(5,2,2,3,3,1)$ and $S=(2,2,5,4,1,1,1)$, respectively.
\[A_1=\left[
\begin{array}{c|c|c|c|c|c|c}
&1&1&1&1&&\\ \hline
&&1&&&&\\ \hline
&&1&&&&\\ \hline
&&&&&1&1\\ \hline
1&&&&&&\\ \hline
1&&&&&&\end{array}\right], 
\qquad
A_2=
\left[
\begin{array}{c|c|c|c|c|c|c}
&&1&1&1&1&1\\ \hline
&&1&1&&&\\ \hline
&&1&1&&&\\ \hline
&1&1&1&&&\\ \hline
1&1&1&&&&\\ \hline
1&&&&&&\end{array}\right].\]
}\hfill{$\Box$}
\end{example}

\begin{example}{\rm \label{ex:order}
This next example stresses that, not surprisingly,  the ordering of the components of $R$ and $S$ is important for the nonemptiness of ${\mathcal C}(R,S)$.
Let $R=S=(2,2,1)$. Then we have the convex matrices
\[\left[\begin{array}{c|c|c}
1&1&\\ \hline
1&1&\\ \hline
&&1\end{array}\right]\mbox{ and }
\left[\begin{array}{c|c|c}
&1&1\\ \hline
1&1&\\ \hline
1&&\end{array}\right]\mbox{ but also the nonconvex }
\left[\begin{array}{c|c|c} 1&&1\\ \hline 1&1&\\ \hline &1&\end{array}\right]. \]
Now let $R=S=(2,1,2)$. Then, as is easily checked,  ${\mathcal C}(R,S)=\emptyset$.
Note that the a matrix whose 1's form a Ferrers diagram (the maximal matrix of Ryser when $R$ and $S$ are nondecreasing) is convex.
}\hfill{$\Box$}\end{example}

As mentioned, it is an {\em NP}-complete problem to determine whether a class $\mc{C}(R,S)$ is nonempty, but a goal of this paper is to characterize this property in terms of $R$ and $S$ for certain subclasses  of $\mc{C}_{m,n}(R,S)$.

An {\it epitope}\footnote{The word `epitope' is a synonym for determinant, a word we avoid in the current context since it has a well-known and different meaning in mathematics.}
 of ${\mathcal A}(R,S)$ is any information that determines uniquely a  matrix that is known to be in ${\mathcal A}(R,S)$. Trivially the positions of all but one of the 1's of an $A\in {\mathcal A}(R,S)$ is an epitope of ${\mathcal A}(R,S)$. But knowing the position of all but two 1's is not an epitope in general, as seen with $R=(1,1)=S$ and the matrices
 \[\left[\begin{array}{cc} 1&0 \\ 0&1\end{array}\right] \mbox{ and }
 \left[\begin{array}{cc} 0&1 \\ 1&0\end{array}\right].\]
 In fact, any two matrices in ${\mc A}(R,S)$ can be obtained from one another by a sequence of {\it interchanges} which replace any $2\times 2$ submatrix equal to one of these with the other.
 
 Let $m=n$ and $R=S=(1,1,\ldots,1)$. Then ${\mathcal A}(R,S)$ is the set ${\mathcal P}_n$ of $n\times n$ permutation matrices.
In \cite{Fulton92} Fulton determined an epitope of ${\mathcal P}_n$  called the  ranked essential set obtained as follows.   For each 1 of $P \in \mc{P}_n$ shade (or delete) all the positions from the 1 and eastwards and from the 1 and southwards leaving the {\it diagram}   of $P$. The {\it essential set} of $P$ is the set of southeast corners of the connected components of the  unshaded squares of the diagram. For example, if 
\[P=\left[\begin{array}{c|c|c|c|c|c}
&1&&&&\\ \hline
&&&&1&\\ \hline
1&&&&&\\ \hline
&&&1&&\\ \hline
&&&&&1\\ \hline
&&1&&&
\end{array}\right],\]
then, using the above description, we have
\begin{equation}\label{eq:P}
P: \left|\begin{array}{c|c|c|c|c|c}
\hline
&\cellcolor[gray]{.9}1&\cellcolor[gray]{0.9}&\cellcolor[gray]{0.9}&\cellcolor[gray]{0.9}&\cellcolor[gray]{0.9}\\ \hline
\star&\cellcolor[gray]{0.9}&&\star&\cellcolor[gray]{0.9} 1&\cellcolor[gray]{0.9}\\ \hline
\cellcolor[gray]{0.9} 1&\cellcolor[gray]{0.9}&\cellcolor[gray]{0.9}&\cellcolor[gray]{0.9}&\cellcolor[gray]{0.9}&\cellcolor[gray]{0.9}  \\ \hline
\cellcolor[gray]{0.9}&\cellcolor[gray]{0.9}&&\cellcolor[gray]{0.9} 1&\cellcolor[gray]{0.9}&\cellcolor[gray]{0.9}\\ \hline
\cellcolor[gray]{0.9}&\cellcolor[gray]{0.9}&\star&\cellcolor[gray]{0.9}&\cellcolor[gray]{0.9}&\cellcolor[gray]{0.9}1\\ \hline
\cellcolor[gray]{0.9}&\cellcolor[gray]{0.9}\cellcolor[gray]{0.9}&\cellcolor[gray]{0.9} 1&\cellcolor[gray]{0.9}&\cellcolor[gray]{0.9}&\cellcolor[gray]{0.9} \\ \hline
 \end{array}\right|
 \rightarrow
 \left|\begin{array}{c|c|c|c|c|c}
\hline
\phantom{\star}&\cellcolor[gray]{.9}\ 1&\cellcolor[gray]{0.9}&\cellcolor[gray]{0.9}&\cellcolor[gray]{0.9}&\cellcolor[gray]{0.9}\\ \hline
0&\cellcolor[gray]{0.9}&&1&\cellcolor[gray]{0.9} 1&\cellcolor[gray]{0.9}\\ \hline
\cellcolor[gray]{0.9} 1&\cellcolor[gray]{0.9}&\cellcolor[gray]{0.9}&\cellcolor[gray]{0.9}&\cellcolor[gray]{0.9}&\cellcolor[gray]{0.9}  \\ \hline
\cellcolor[gray]{0.9}&\cellcolor[gray]{0.9}&&\cellcolor[gray]{0.9}1&\cellcolor[gray]{0.9}&\cellcolor[gray]{0.9}\\ \hline
\cellcolor[gray]{0.9}&\cellcolor[gray]{0.9}&2&\cellcolor[gray]{0.9}&\cellcolor[gray]{0.9}&\cellcolor[gray]{0.9} 1\\ \hline
\cellcolor[gray]{0.9}&\cellcolor[gray]{0.9}\cellcolor[gray]{0.9}&\cellcolor[gray]{0.9} 1&\cellcolor[gray]{0.9}&\cellcolor[gray]{0.9}&\cellcolor[gray]{0.9} \\ \hline
 \end{array}\right|,\end{equation}
 where the essential set consists of the squares with a $\star$ which are then replaced by the number of $1$'s in the northwest submatrix they determine (equivalently, the  rank of the northwest submatrix); recall that the shaded squares are considered deleted.
 Fulton shows that a permutation matrix is determined by these rank numbers and their locations, the {\it ranked essential set}, and that in general none of these can be omitted. For emphasis we remark that the ranked essential set determines $P$ with the knowledge  that $P$ is an $n\times n$ permutation matrix, that is, an $n\times n$ $(0,1)$-matrix with exactly one 1 in each row and column.  The ranked essential set  in (\ref{eq:P}) is given by 
 \[\{(2,1;0), (2,4;1), (5,3;2)\}\]
 where the first two entries in each triple specify the location and the last entry specifies the corresponding rank. These triples  determine the  locations of the 1's of the permutation matrix $P$ and thus {\it the ranked essential set is an epitope for the classes ${\mathcal A}(R,S)$ when $R=S=(1,1,\ldots,1)$}.
 In \cite{EL96} an algorithm is given that determines a permutation matrix from its ranked essential set.  We show that this algorithm works also for a collection of  ``convex matrices" which we call generalized polyominos. In \cite{EL96} it is also shown that a permutation matrix is determined by the rank function on a subset of its essential set called its {\it core} where the core has size at most $n$. Thus the core is also an epitope for these classes. In \cite{EL95} it is proved that  the average size of the essential set of an $n\times n$ permutation matrix is asymptotic to $\frac{n^2}{36}$. The only permutation matrix in ${\mathcal P}_n$ with an empty essential set is the identity matrix. If the essential set has size more than $n$, then the permutation matrix itself is a more compact representation. Therefore, classes of permutation matrices with ``small" ranked essential sets are interesting.

Recall that a permutation $\pi$ of $\{1,2,\ldots,n\}$ is a {\it grassmannian} provided it has exactly one descent (i.e., a pair $(\pi_i,\pi_{i+1})$ such that $\pi_i> \pi_{i+1}$) and is a {\it bigrassmannian} provided  both the  permutation and its inverse are grassmannians.
 With $n=8$, the permutation $(2,5,7,1,3,4,6,8)$ has exactly one descent, namely $(7,1)$ and is a grassmannian.  Its inverse  $(4,1,5,6,2,7,3,8)$  has three descents, namely $(4,1)$, $(6,2)$, and $(7,3)$, and thus   is not a bigrassmannian. The permutation
$(4,5,1,2,3)$ is a bigrassmanian as it has exactly one descent, namely $(5,1)$, and its inverse $(3,4,5,1,2)$ has exactly one descent, namely $(5,1)$. We apply  the words grassmannian and bigrassmannian to permutation matrices as well.

\begin{example}\label{ex:one}{\rm Consider the grassmannian permutation $(2,5,7,1,3,4,6,8)$,  its corresponding grassmannian permutation matrix, its diagram, and its ranked essential set:
\[\begin{array}{|c|c|c|c|c|c|c|c|}
\hline
&\cellcolor[gray]{0.9}1&\cellcolor[gray]{0.9}&\cellcolor[gray]{0.9}&\cellcolor[gray]{0.9}&\cellcolor[gray]{0.9}&\cellcolor[gray]{0.9}&\cellcolor[gray]{0.9}\\ \hline
&\cellcolor[gray]{0.9}&&&\cellcolor[gray]{0.9}1&\cellcolor[gray]{0.9}&\cellcolor[gray]{0.9}&\cellcolor[gray]{0.9}\\ \hline
0&\cellcolor[gray]{0.9}&&1&\cellcolor[gray]{0.9}&2
&\cellcolor[gray]{0.9}1&\cellcolor[gray]{0.9}\\ \hline
\cellcolor[gray]{0.9}1\cellcolor[gray]{0.9}&\cellcolor[gray]{0.9}&\cellcolor[gray]{0.9}&\cellcolor[gray]{0.9}&\cellcolor[gray]{0.9}&\cellcolor[gray]{0.9}&
\cellcolor[gray]{0.9}&\cellcolor[gray]{0.9}\\ \hline
\cellcolor[gray]{0.9}&\cellcolor[gray]{0.9}&\cellcolor[gray]{0.9}1&\cellcolor[gray]{0.9}&\cellcolor[gray]{0.9}&\cellcolor[gray]{0.9}&\cellcolor[gray]{0.9}\cellcolor[gray]{0.9}&\cellcolor[gray]{0.9}\\ \hline
\cellcolor[gray]{0.9}&\cellcolor[gray]{0.9}&\cellcolor[gray]{0.9}&\cellcolor[gray]{0.9}1&\cellcolor[gray]{0.9}&\cellcolor[gray]{0.9}&\cellcolor[gray]{0.9}&\cellcolor[gray]{0.9}\\ \hline
	\cellcolor[gray]{0.9}&\cellcolor[gray]{0.9}&\cellcolor[gray]{0.9}&\cellcolor[gray]{0.9}&\cellcolor[gray]{0.9}&\cellcolor[gray]{0.9}1&\cellcolor[gray]{0.9}&\cellcolor[gray]{0.9}\\ \hline
\cellcolor[gray]{0.9}&\cellcolor[gray]{0.9}&\cellcolor[gray]{0.9}&\cellcolor[gray]{0.9}&\cellcolor[gray]{0.9}&\cellcolor[gray]{0.9}&\cellcolor[gray]{0.9}&\cellcolor[gray]{0.9}1\\ \hline
\end{array}.\]
The ranked essential set is $\{(3,1;0),(3,4;1), (3,6;2)\}$. Notice that the essential set is contained in one row, thus is  {\it row-linear}, and the rank numbers in that row are strictly increasing.
}\hfill{$\Box$}
\end{example}

In \cite{Kobayashi10} it is shown that bigrassmannians are characterized by having their essential sets of cardinality 1. In \cite{EL96} it is shown that grassmannians are characterized by having a row-linear essential set and so we get again that bigrassmannians are characterized by having their essential sets of cardinality 1 (apply the transpose  operation to get the permutation matrix corresponding to the inverse).

The remaining paper is organized as follows. Section \ref{sec:convex} considers convex sets and extends the definition of ranked essential sets to that class. It is shown that the ranked essential set determines a matrix in $\mc{C}(R,S)$ uniquely, and an algorithm for doing this is given. 
Next, Section \ref{sec:subclass-character} is devoted to  subclasses of $\mc{C}(R,S)$ and characterizations of when these classes are nonempty. Section \ref{sec:int-convex} concerns interchanges that preserve convexity and $\mc{A}(R,S)$ classes with only convex matrices. 
Finally, in Section \ref{sec:poset}, we discuss  a natural partial order on the class of $m \times n$ convex matrices. 

Notation: By a {\it line} of a matrix, we mean either a row or a column. The $k\times l$ matrix of all 1's is denoted by $J_{k,l}$, abbreviated to $J_k$ if $l=k$.

\section{Convex $(0,1)$-matrices and ranked essential sets}
\label{sec:convex}

We discuss convex matrices and introduce their ranked essential sets.
A $(0,1)$-matrix $A$ is called {\it connected}
provided that 

(i) it does not have any zero rows or zero columns, and  

(ii) every pair of 1's is  connected by a rookwise path of either row adjacent or column adjacent 1's.

\noindent
To say that a $(0,1)$-matrix is connected  is equivalent to saying that it is the biadjacency matrix of a connected bipartite graph, in particular,  it does not have any isolated vertices. We now show that the essential set  is an epitope of a polyomino in ${\mathcal A}(R,S)$.  We refer to \cite{Balazs07} for information about convex $(0,1)$-matrices for these and the definitions to follow.  A {\it polyomino} is a connected, convex $(0,1)$-matrix\footnote{We could allow zero rows and zero columns in a polyomino as is often done, but such rows and columns would have to be initial or terminal, and thus would not play any role in our investigations.}.  A {\it southeast source}, abbreviated to {\it SE-source}, of a $(0,1)$-matrix is a 1  such that every other 1 is reachable from it by a southeast rookwise path of row adjacent or column adjacent 1's.  NE-sources, NW-sources, and SW-sources are define analogously. When such sources exist, they are clearly unique. The matrix $A$ is {\it directed} provided it has a source of at least one of these types. Since under 90 degree rotations, a source of one type becomes a source of another type, without loss of generality one may assume that a source is a SE-source.  If $A$ is directed and does not have any zero rows or zero columns, then clearly $A$ is connected.
A connected $(0,1)$-matrix $A$ does not have zero rows or zero columns, and hence if $A$ has a SE-source, it must be a 1 in position $(1,1)$.

The first convex matrix in Example \ref{ex:convex0} is neither connected nor directed; the second is a directed polyomino with a NW-source and a SW-source

The  following  theorem is from \cite{Kuba} (see also \cite{Balazs07}).

\begin{theorem}  \label{lem:balazs} A directed, convex matrix $A$ in ${\mathcal A}(R,S)$ is uniquely determined by its source and can be reconstructed in $O(mn)$ time. In particular, a polyomino known to be in a class ${\mathcal A}(R,S)$ which has  a $1$ in position $(1,1)$ $($or $(1,n)$ or $(m,1)$ or $(m,n)$$)$ is uniquely determined.
\end{theorem}

\begin{example}\label{ex:source}{\rm  If the row and column sum vectors $R$ and $S$  do not have any zeros, then the only possibility for a source is a 1 in one of the corner positions, and each such 1 is a source. Let $R=(1,3,4,3,1)$ and $S=(1,2,4,3,2)$. Then the two matrices
\[A_1=\left[\begin{array}{c|c|c|c|c}
&&&1&\\ \hline
&&1&1&1\\ \hline
&1&1&1&1\\ \hline
1&1&1&&\\ \hline
&&1&&\end{array}\right]\mbox{ and }
A_2=\left[\begin{array}{c|c|c|c|c}
&&1&&\\ \hline
&1&1&1&\\ \hline
1&1&1&1&\\ \hline
&&1&1&1\\ \hline
&&&&1\end{array}\right]\]
both belong to ${\mc C}(R,S)$ where $A_1$ does not have a source but $A_2$ has a NW-source.
Thus $A_2$ is uniquely constructable from $R$ and $S$ and its source in position $(5,5)$.
}\endproof\end{example}


We can define the essential set and ranked essential set  for any $m\times n$ $(0,1)$-matrix $A$ as for permutation matrices: For each 1 of $A$  shade (or cross out) its row to the east and its column to the
south (including the 1 itself). This gives the {\it diagram} of $A$, defined as the set of unshaded positions. The {\it essential set} of $A$ is the set of southeast corners of the connected components of the diagram of $A$.
In each element of the essential set  put the number of 1's in the leading submatrix it determines (that is, the number of $1$'s in northeast submatrix with the specified element of the essential set in its southeast corner). These corner positions and their numbers give the {\it ranked essential set}.

\begin{example}\label{ex:convex2} {\rm Consider the polyomino $A_2$ in Example \ref{ex:convex0} with row sum vector $R=(5,2,2,3,3,1)$ and column sum vector $S=(2,2,5,4,1,1,1)$. Then its ranked essential set is 
easily computed to be $\{(3,2;0),(4,1;0)\}$. Any  convex matrix in ${\mathcal C}(R,S)$ with this ranked essential set equals $A_2$; see below.
}\hfill{$\Box$}\end{example}

We define a {\em Ferrers array} $F$ to be a left-justified array of positions  where  the number of positions in the rows form   a nonincreasing  vector $U$. Thus, $F$ is uniquely defined by $U$, and we write $F=F(U)$.

\begin{theorem} 
\label{lem:epi}
 Let $A$ be a polyomino of size $m \times n$. Then the diagram of  $A$ is a  Ferrers array F. Its ranked essential set consists of the positions of the lower right corners of  $F$ each filled with a $0$. 
A polyomino  $A\in {\mc C}_{m,n}(R,S)$  is uniquely determined by its ranked essential set.
\end{theorem}

\begin{proof}  
Let $A=[a_{ij}]$. Since every row and column contains a 1, there exist  a smallest $i$ and a smallest $j$ such that $a_{i1}=a_{1j}=1$. Also, as $A$ is connected and convex, the leftmost 1 in the first $i$ rows determine a northeast rookwise path $P_1$ from $(i,1)$ to $(1,j)$ (possibly, it only consists of $(1,1)$). There is also a smallest $k$ such that $a_{mk}=1$ and  a southeast rookwise path $P_2$ from $(i,1)$ to  position $(m,k)$. Clearly, the only positions in $A$ that remain after crossing out the row to the east and column to the south of the 1's in  $P_1 \cup P_2$   are the positions northwest of the path $P_1$, and these positions form a Ferrers array $F$ and give the diagram of $A$.  The lower right corners of $F$ constitute the essential set  of $A$, and the ranked essential set has 0's  in all of these positions. Thus  all the entries of $A$ in the positions of the diagram $F$ equal 0. Moreover, the essential set also determines the diagram $F$. Then each position in $A$ adjacent to $F$ in a row or column is a 1. Now  with convexity, connectivity, and $R$ and $S$, the remaining positions  of $A$ are  determined.
\end{proof}

\begin{corollary} \label{cor:new}
The ranked essential set is an epitope  of a polyomino  in ${\mathcal A}(R,S)$.
\end{corollary}

\begin{corollary}\label{cor:new2}
The essential set of a polyomino $A$ is empty if and only if $A$ has a SE-source.
\end{corollary}

\begin{proof} If the essential set of $A$ is empty, then clearly $A$ has a 1 in position $(1,1)$ and thus a SE-source. Conversely, assume that $A$ has a SE-source which is necessarily a 1 in position $(1,1)$. Then the associate Ferrers array, that is, its essential set  is empty.
\end{proof}

\begin{example}
\label{ex:convex1}
 {\rm Below is the matrix $A_2$ from Example \ref{ex:convex0} and its ranked essential set (to the left); the diagram is the unshaded area constituting a Ferrers array $F$. The matrix $A_2$ is uniquely reconstructed by first putting zeros in all the positions in the Ferrers array,  and then putting ones in  all positions adjacent to $F$. The row and column sums then determine $A_2$. Here  $R=(5,2,2,3,3,1)$ and $S=(2,2,5,4,1,1,1)$.
\[\left[
\begin{array}{c|c|c|c|c|c|c}
&&\cellcolor[gray]{0.9}&\cellcolor[gray]{0.9}&\cellcolor[gray]{0.9}&\cellcolor[gray]{0.9}&\cellcolor[gray]{0.9}\\ \hline
&&\cellcolor[gray]{0.9}&\cellcolor[gray]{0.9}&\cellcolor[gray]{0.9}&\cellcolor[gray]{0.9}&\cellcolor[gray]{0.9}\\ \hline
&0&\cellcolor[gray]{0.9}\mbox{ }&\cellcolor[gray]{0.9}\mbox{ }&\cellcolor[gray]{0.9}\mbox{ }&\cellcolor[gray]{0.9}\mbox{ }&\cellcolor[gray]{0.9}\mbox{ }\\ \hline
0&\cellcolor[gray]{0.9}&\cellcolor[gray]{0.9}&\cellcolor[gray]{0.9}&\cellcolor[gray]{0.9}&\cellcolor[gray]{0.9}&\cellcolor[gray]{0.9}\\ \hline
\cellcolor[gray]{0.9}&\cellcolor[gray]{0.9}&\cellcolor[gray]{0.9}&\cellcolor[gray]{0.9}&\cellcolor[gray]{0.9}&\cellcolor[gray]{0.9}&\cellcolor[gray]{0.9}\\ \hline
\cellcolor[gray]{0.9}&\cellcolor[gray]{0.9}&\cellcolor[gray]{0.9}&\cellcolor[gray]{0.9}&\cellcolor[gray]{0.9}&\cellcolor[gray]{0.9}&\cellcolor[gray]{0.9}
\end{array}\right]\rightarrow
\left[
\begin{array}{c|c|c|c|c|c|c}
0&0&1&\phantom{1}&\phantom{1}&\phantom{1}&\phantom{1}\\ \hline
0&0&1&&&&\\ \hline
0&0&1&&&&\\ \hline
0&1&&&&&\\ \hline
1&&&&&&\\ \hline
&&&&&&\end{array}\right]\rightarrow
\left[
\begin{array}{c|c|c|c|c|c|c}
0&0&1&1&1&1&1\\ \hline
0&0&1&1&&&\\ \hline
0&0&1&1&&&\\ \hline
0&1&1&1&&&\\ \hline
1&1&1&&&&\\ \hline
1&&&&&&\end{array}\right].\]
}\hfill{$\Box$}
\end{example}

We define a {\it generalized polyomino} to be a convex $m\times n$ $(0,1)$-matrix without zero rows or zero columns. Thus, in contrast to a polyomino,  a   generalized polyomino may not be connected. A generalized polyomino is  obtained from some $k\times k$  permutation matrix by replacing each of its 1's by a polyomino. If the 1 in row $i$ of $P$ is replaced by the polyomino $A_i$ $(1\le i\le k)$, then we denote this generalized polyomino by $P(A_1,A_2,\ldots,A_k)$.

\begin{example}\label{ex:genpoly1} {\rm Let 
\[P=\left[\begin{array}{ccc}0 &1&0\\ 0&0&1\\ 1&0&0\end{array}\right],
A_1=\left[\begin{array}{ccc} 0&1&1\\  1&1&0\\ 1&0&0\end{array}\right],
A_2=\left[\begin{array}{ccc} 0&1&1\\ 1&1&1\\ 1&0&0 \end{array}\right],
A_3=\left[\begin{array}{ccc}  1&1&1\\ 0&1&1\end{array}\right].\]
Then
\[P(A_1,A_2,A_3)=
\left[\begin{array}{ccc||ccc||ccc}
&&&0&1&1&&&\\
&&&1&1&0&&&\\
&&&1&0&0&&&\\ \hline\hline
&&&&&&0&1&1\\
&&&&&&1&1&1\\
&&&&&&1&0&0\\  \hline\hline
0&1&1&&&&&&\\
1&1&1&&&&&&\end{array}\right]\]
where unspecified entries equal 0.
}\endproof \end{example}
\noindent
 Our generalized polyomino is  an 8-connected $(0,1)$-convex matrix  as defined  in
  \cite{Balazs07},  but without any zero rows or columns.
In the lemma below, for ease of exposition and subsequent examples, we now consider the diagram of an $m\times n$  $(0,1)$-matrix $A$ to be an $m\times n$ array (rather than just a subset of positions of an array) whose unshaded positions are what we called previously the diagram of $A$; so the diagram is considered as an $m\times n$ array of positions some of which are unshaded.
The following lemma is straightforward to verify.

\begin{lemma}\label{lem:diagram}
The diagram of the generalized polyomino $P(A_1,A_2,\ldots,A_k)$ is obtained from the diagram of  $P$  by replacing its shaded zeros with appropriately sized shaded arrays, its unshaded zeros with appropriately unshaded arrays, and its ones with the corresponding  diagram of the $A_i$'s.
\end{lemma}

\begin{example}\label{ex:genpoly2} {\rm Continuing with Example \ref{ex:genpoly1}, the diagram of $P(A_1,A_2,A_3)$ and its ranked essential set is as specified in
\[\left[\begin{array}{ccc||ccc||ccc}
&&&0&\cellcolor[gray]{0.8}&\cellcolor[gray]{0.8}&&&\\ 
&&&\cellcolor[gray]{0.8}&\cellcolor[gray]{0.8}&\cellcolor[gray]{0.8}&&\cellcolor[gray]{0.8}X&\\ 
&&&\cellcolor[gray]{0.8}&\cellcolor[gray]{0.8}&\cellcolor[gray]{0.8}&&&\\  \hline\hline
&&&&&&5&\cellcolor[gray]{0.8}&\cellcolor[gray]{0.8}\\ 
&&&&\cellcolor[gray]{0.8}X&&\cellcolor[gray]{0.8}&\cellcolor[gray]{0.8}&\cellcolor[gray]{0.8}\\ 
&&0&&&&\cellcolor[gray]{0.8}&\cellcolor[gray]{0.8}&\cellcolor[gray]{0.8}\\ \hline\hline
0&\cellcolor[gray]{0.8}&\cellcolor[gray]{0.8}&&&&&&\\ 
\cellcolor[gray]{0.8}&\cellcolor[gray]{0.8}&\cellcolor[gray]{0.8}&&\cellcolor[gray]{0.8}X&&&\cellcolor[gray]{0.8}X&
\end{array}\right],\]
where the shaded $X$'s are matrices of appropriate sizes with all other positions shaded.
}
\endproof\end{example}


\begin{example}\label{ex:convexconstr}
{\rm Consider the generalized polyomino $P(A_1,A_2)$, with  $P=I_2$, $R=(2,2,2,1,1,2)$, and $S=(2,2,1,4,1)$ given by
\[A=
\left[\begin{array}{c|c|c|c|c}
1&1&&&\\ \hline
1&1&&&\\ \hline
&&&1&1\\ \hline
&&&1&\\ \hline
&&&1&\\ \hline
&&1&1&\end{array}\right].\]
Bordering the matrix  by its row sum vector and column sum vector for visual ease, we see that the essential set and ranked essential set   are specified in 
\[\left[\begin{array}{c||c|c|c|c|c}
&2&2&1&4&1\\ \hline\hline
2&\cellcolor[gray]{0.8}&\cellcolor[gray]{0.8}&\cellcolor[gray]{0.8}&\cellcolor[gray]{0.8}&\cellcolor[gray]{0.8}\\ \hline
2&\cellcolor[gray]{0.8}&\cellcolor[gray]{0.8}&\cellcolor[gray]{0.8}&\cellcolor[gray]{0.8}&\cellcolor[gray]{0.8}\\ \hline
2&\cellcolor[gray]{0.8}&\cellcolor[gray]{0.8}&&\cellcolor[gray]{0.8}&\cellcolor[gray]{0.8}\\ \hline
1&\cellcolor[gray]{0.8}&\cellcolor[gray]{0.8}&&\cellcolor[gray]{0.8}&\cellcolor[gray]{0.8}\\ \hline
1&\cellcolor[gray]{0.8}&\cellcolor[gray]{0.8}&4&\cellcolor[gray]{0.8} &\cellcolor[gray]{0.8}\\ \hline
2&\cellcolor[gray]{0.8}&\cellcolor[gray]{0.8}&\cellcolor[gray]{0.8}&\cellcolor[gray]{0.8}&\cellcolor[gray]{0.8}\end{array}\right].\]
We show how  $A$ can be constructed from $R$, $S$, and the ranked essential set.
Using convexity and  the ranked essential set as determined above, it is not hard to first get
\[\left[\begin{array}{c||c|c|c|c|c}
&2&2&1&4&1\\ \hline\hline
2&1&1&0&&\\  \hline
2&1&1&0&&\\ \hline
2&0&0&0&&\\ \hline
1&0&0&0&&\\ \hline
1&0&0&0&&\\ \hline
2&&&&&\end{array}\right]
\mbox{ and then }
\left[\begin{array}{c||c|c|c|c|c}
&2&2&1&4&1\\ \hline\hline
2&1&1&0&0&0\\  \hline
2&1&1&0&0&0\\ \hline
2&0&0&0&1&1\\ \hline
1&0&0&0&1&0\\ \hline
1&0&0&0&1&0\\ \hline
2&0&0&1&1&0\end{array}\right].\]

As another example, the generalized polyomino with specified ranked essential set given by
\[\left[\begin{array}{c||c|c|c|c|c|c|c} 
&1&1&3&2&2&2&1\\ \hline\hline
1&&&1&&&&\\ \hline
1&&&1&&&&\\ \hline
3&1&1&1&&&&\\ \hline
2&&&&&&1&1\\ \hline
2&&&&&1&1&\\ \hline
2&&&&1&1&&\\ \hline
1&&&&1&&&\end{array}\right]\rightarrow
\left[\begin{array}{c||c|c|c|c|c|c|c} 
&1&1&3&2&2&2&1\\ \hline\hline
1&&&\cellcolor[gray]{0.8}&\cellcolor[gray]{0.8}&\cellcolor[gray]{0.8}&\cellcolor[gray]{0.8}&\cellcolor[gray]{0.8}\\ \hline
1&&0&\cellcolor[gray]{0.8}\cellcolor[gray]{0.8}&\cellcolor[gray]{0.8}\cellcolor[gray]{0.8}&\cellcolor[gray]{0.8}
&\cellcolor[gray]{0.8}&\cellcolor[gray]{0.8}\\ \hline
3&\cellcolor[gray]{0.8}&\cellcolor[gray]{0.8}&\cellcolor[gray]{0.8}&\cellcolor[gray]{0.8}&\cellcolor[gray]{0.8}&\cellcolor[gray]{0.8}&\cellcolor[gray]{0.8}\\ \hline
2&\cellcolor[gray]{0.8}&\cellcolor[gray]{0.8}&\cellcolor[gray]{0.8}&&5&\cellcolor[gray]{0.8}&\cellcolor[gray]{0.8}\\ \hline
2&\cellcolor[gray]{0.8}&\cellcolor[gray]{0.8}&\cellcolor[gray]{0.8}&5&\cellcolor[gray]{0.8}&\cellcolor[gray]{0.8}&\cellcolor[gray]{0.8}\\ \hline
2&\cellcolor[gray]{0.8}&\cellcolor[gray]{0.8}&\cellcolor[gray]{0.8}&\cellcolor[gray]{0.8}&\cellcolor[gray]{0.8}&\cellcolor[gray]{0.8}&\cellcolor[gray]{0.8}\\ \hline
1&\cellcolor[gray]{0.8}&\cellcolor[gray]{0.8}&\cellcolor[gray]{0.8}&\cellcolor[gray]{0.8}&\cellcolor[gray]{0.8}&\cellcolor[gray]{0.8}&\cellcolor[gray]{0.8}\end{array}\right]
\]
is readily reconstructable.
}\hfill{$\Box$}\end{example}



We now extend the algorithm given in \cite{EL96} (for permutation matrices) to generalized polyominoes.
Let $A$ be an $m\times n$ generalized polyomino in ${\mc C}(R,S)$. We use the {\it entrywise partial order} on positions 
defined by $(i,j)\le (k,l)$ provided $i\le k$ and $j\le l$.

\bigskip
\centerline{\bf Algorithm to reconstruct a generalized polyomino in ${\mathcal C}(R,S)$}
\centerline{\bf from its ranked essential set}
\begin{itemize}
\item[0.] Input: The nonnegative, integral vectors $R$, $S$ (with the same sum) and a candidate ranked essential set, i.e., some set (possibly empty) of positions of an $m \times n$ matrix $E$ and a corresponding nonnegative number in each of these positions, and an $m\times n$ matrix $A$  with all positions initially empty.

\item[1.]  Alternate between applying the following steps below to $E$ and ${A}$ until either all entries of
 ${A}$ have been determined,   or  some entry of  ${E}$ in a position of the essential set has become negative.
 \begin{itemize}
 \item[1a.] For each position ($i,j)$ in the essential set equal to 0, put a 0 in each position $(k,l)$ with $(k,l)\le (i,j)$ and then remove those positions from further consideration, including the positions of the essential set containing a 0. 
 \item[1b.] Put a 1 in a  minimal (in the partial order) position $(p,q)$ of ${A}$ and decrease by 1 the remaining positions $(u,v)$ of the essential set  with $(u,v)\ge (p,q)$.   Repeat until all minimal positions have a 1. If a row or column  has the correct sum, fill in the remaining positions of that row or column with zeros.
 \end{itemize}
\item[2.] If some entry of ${E}$ has become negative, the given set is not a ranked essential set of a generalized polyomino; otherwise, output ${A}$.
\end{itemize}

The following theorem generalizes Theorem \ref{lem:balazs}.

\begin{theorem} 
\label{lem:alg-works}
 The function that maps  a matrix in $\mc{C}_{m,n}(R,S)$ into its ranked essential set is injective. Moreover, the algorithm above reconstructs $A \in \mc{C}_{m,n}(R,S)$ from its ranked essential set in $O(mn)$ time.
\end{theorem}

\begin{proof}  
The proof is by induction on the number $m+n$ of lines. Let $f$ denote the function defined on $\mc{C}_{m,n}(R,S)$ that maps a matrix $A \in \mc{C}_{m,n}(R,S)$ into its ranked essential set $f(A)=E$.

Consider first the case when $E$ is empty. Let $A =[a_{ij}] \in f^{-1}(E)$. Then $a_{11}=1$ (otherwise $(1,1)$ would be in the diagram and $E$ would be nonempty), and the first row and column are determined by convexity ($r_1$ and $s_1$ ones, respectively, followed by zeros). Then the first $r_1$ columns and the first $s_1$ rows are also determined by convexity. We continue like this until $A$ is partially determined, uniquely, as a direct sum $A=A_1 \oplus A_2$. Here $A_2$ has smaller size than $A$, and by induction it is uniquely determined by its ranked essential set, which must be empty. Therefore $A$ is uniquely determined, i.e., $|f^{-1}(E)|=1$, as desired. (We see that, in the case of an empty essential set of a generalized polyomino $A$,
the corresponding permutation matrix $P$ is an identity matrix.)

Next, assume $E$ is nonempty. Let $(k,l)$ be a minimal element in the entrywise partial order of our ranked essential set and let $r$ be its rank.  We consider two cases, namely $r\ge 1$ and $r=0$. 

If $r \ge 1$, then $a_{11}=1$ and we can argue as in the previous paragraph and uniquely fill in ones and zeros by convexity. Note that each such 1 must be in some position $(i,j)$ where $i<k$ and $j<l$ because any matrix $A$ with $E=f(A)$ has zeros in row $k$ and column $l$ up to position $(k,l)$. Therefore $A$ is a direct sum $A=A_1 \oplus A_2$ and we are done by induction as in the previous paragraph. Note that the rank associated with each position of the essential set in the region occupied  by $A_2$ is  reduced by $r$.

Next, consider the case $r=0$. Let $F$ be the set of positions $(i,j) \le (p,q)$ for some minimal position $(p,q)$ in the minimal essential set with rank $0$. Then $F$ is  the set of positions of a Ferrers array, and $a_{ij}=0$ for every $(i,j) \in F$ and every $A=[a_{ij}] \in f^{-1}(E)$. Next, all the positions adjacent to $F$ are uniquely determined by convexity and line sums. Then $A$ is a direct sum and we proceed by induction as above.
This proves that the function $f$ is injective, and the first part of the theorem holds. 

The correctness of the algorithm is seen from the arguments above, because the uniquely determined parts of the matrix are given values as in the algorithm. So, if $E=f(A)$ for some $A \in \mc{C}_{m,n}(R,S)$, the algorithms computes $A$. If the algorithm stops and does not return any matrix in $\mc{C}_{m,n}(R,S)$, no such matrix exists for the given set $E$. That the number of steps is $O(mn)$ is clear from the algorithm.
\end{proof}

\begin{example}{\rm \label{ex:algorithm} We apply the algorithm to the generalized polyomino $P(A_1,A_2,A_3)$ 
\[\left[\begin{array}{cc|cc|cccc}
&&1&1&&&&\\ 
&&0&1&&&&\\ 
&&0&1&&&&\\  \hline
&&&&0&1&0&0\\ 
&&&&1&1&0&0\\ 
&&&&0&1&1&1\\ \hline
0&1&&&&&&\\ 
1&1&&&&&&\end{array}\right],\quad \mbox{ where }
P=\left[\begin{array}{ccc}
0&1&0\\ 0&0&1\\ 1&0&0\end{array}\right].\]
Note that we do not assume that $P$ is known, only that we are dealing with  a generalized polyomino  with  known row sum vector $R=(2,1,1,1,2,3,1,2)$ and known column sum vector  $S=(1,2,1,3,1,3,1,1)$. The ranked essential set is calculated to be as specified in the initial matrix below; now it is convenient to shade the positions of the essential set $E$. To simplify things, we superimpose $E$ on $A$, and when a position and rank of the ranked essential set has been fully taken into account, we remove its shading.
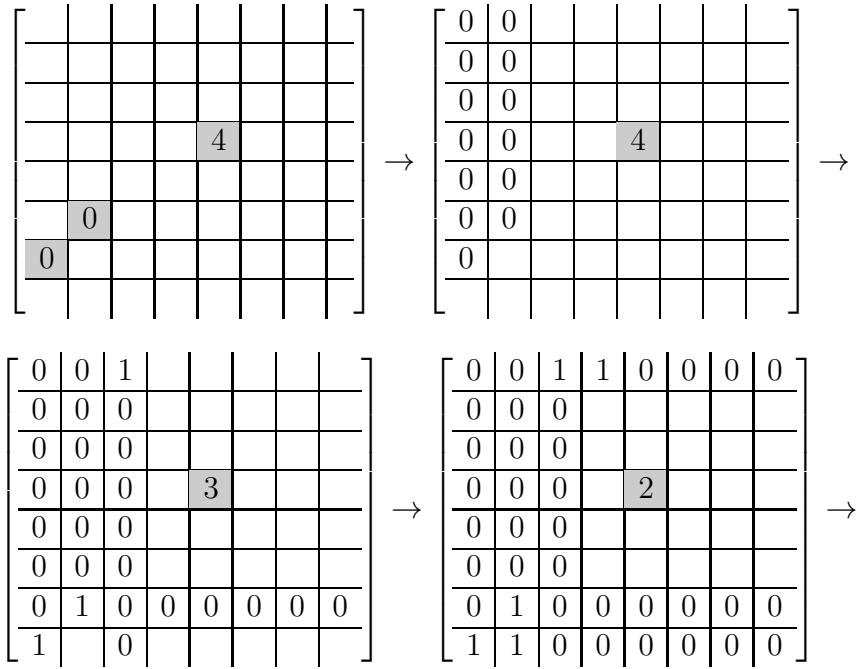
\[
\left[\begin{array}{c|c|c|c|c|c|c|c}
&&\phantom{1}&\phantom{1}&\phantom{1}&\phantom{1}&\phantom{1}\\ \hline
&&&&&&\\ \hline
&&&&&&\\  \hline
&&&&\cellcolor[gray]{0.8}4&&\\ \hline
&&&&&&&\\ \hline
&\cellcolor[gray]{0.8}0&&&&&\\ \hline
\cellcolor[gray]{0.8}0&&&&&&\\ \hline
&&&&&&\end{array}\right]\rightarrow 
\left[\begin{array}{c|c|c|c|c|c|c|c}
0&0&\phantom{1}&\phantom{1}&\phantom{1}&\phantom{1}&\phantom{1}&\phantom{1}\\ \hline
0&0&&&&&&\\ \hline
0&0&&&&&&\\  \hline
0&0&&&\cellcolor[gray]{0.8}4&&&\\ \hline
0&0&&&&&&\\ \hline
0&0&&&&&&\\ \hline
0&&&&&&&\\ \hline
&&&&&&&\end{array}\right]
\rightarrow\]
\[\left[\begin{array}{c|c|c|c|c|c|c|c}
0&0&1&&&&&\\ \hline
0&0&0&&&&&\\ \hline
0&0&0&&&&&\\  \hline
0&0&0&&\cellcolor[gray]{0.8}3&&&\\ \hline
0&0&0&&&&&\\ \hline
0&0&0&&&&&\\ \hline
0&1&0&0&0&0&0&0\\ \hline
1&&0&&&&&\end{array}\right]
\rightarrow
\left[\begin{array}{c|c|c|c|c|c|c|c}
0&0&1&1&0&0&0&0\\ \hline
0&0&0&&&&&\\ \hline
0&0&0&&&&&\\  \hline
0&0&0&&\cellcolor[gray]{0.8}2&&&\\ \hline
0&0&0&&&&&\\ \hline
0&0&0&&&&&\\ \hline
0&1&0&0&0&0&0&0\\ \hline
1&1&0&0&0&0&0&0\end{array}\right]\rightarrow
 \]
 \[
\left[\begin{array}{c|c|c|c|c|c|c|c}
0&0&1&1&0&0&0&0\\ \hline
0&0&0&1&0&0&0&0\\ \hline
0&0&0&&&&&\\  \hline
0&0&0&&\cellcolor[gray]{0.8}1&&&\\ \hline
0&0&0&&&&&\\ \hline
0&0&0&&&&&\\ \hline
0&1&0&0&0&0&0&0\\ \hline
1&1&0&0&0&0&0&0\end{array}\right]\rightarrow
\left[\begin{array}{c|c|c|c|c|c|c|c}
0&0&1&1&0&0&0&0\\ \hline
0&0&0&1&0&0&0&0\\ \hline
0&0&0&1&0&0&0&0\\  \hline
0&0&0&0&\cellcolor[gray]{0.8}0&&&\\ \hline
0&0&0&0&&&&\\ \hline
0&0&0&0&&&&\\ \hline
0&1&0&0&0&0&0&0\\ \hline
1&1&0&0&0&0&0&0\end{array}\right]\rightarrow\]
\[
\left[\begin{array}{c|c|c|c|c|c|c|c}
0&0&1&1&0&0&0&0\\ \hline
0&0&0&1&0&0&0&0\\ \hline
0&0&0&1&0&0&0&0\\  \hline
0&0&0&0&0&1&0&0\\ \hline
0&0&0&0&1&&&\\ \hline
0&0&0&0&0&&&\\ \hline
0&1&0&0&0&0&0&0\\ \hline
1&1&0&0&0&0&0&0\end{array}\right]\rightarrow
\left[\begin{array}{c|c|c|c|c|c|c|c}
0&0&1&1&0&0&0&0\\ \hline
0&0&0&1&0&0&0&0\\ \hline
0&0&0&1&0&0&0&0\\  \hline
0&0&0&0&0&1&0&0\\ \hline
0&0&0&0&1&&&\\ \hline
0&0&0&0&0&&&\\ \hline
0&1&0&0&0&0&0&0\\ \hline
1&1&0&0&0&0&0&0\end{array}\right]\rightarrow \]
\[
\left[\begin{array}{c|c|c|c|c|c|c|c}
0&0&1&1&0&0&0&0\\ \hline
0&0&0&1&0&0&0&0\\ \hline
0&0&0&1&0&0&0&0\\  \hline
0&0&0&0&0&1&0&0\\ \hline
0&0&0&0&1&1&0&0\\ \hline
0&0&0&0&0&&&\\ \hline
0&1&0&0&0&0&0&0\\ \hline
1&1&0&0&0&0&0&0\end{array}\right]\rightarrow
\left[\begin{array}{c|c|c|c|c|c|c|c}
0&0&1&1&0&0&0&0\\ \hline
0&0&0&1&0&0&0&0\\ \hline
0&0&0&1&0&0&0&0\\  \hline
0&0&0&0&0&1&0&0\\ \hline
0&0&0&0&1&1&0&0\\ \hline
0&0&0&0&0&1&&\\ \hline
0&1&0&0&0&0&0&0\\ \hline
1&1&0&0&0&0&0&0\end{array}\right]\rightarrow\]
\[
\left[\begin{array}{c|c|c|c|c|c|c|c}
0&0&1&1&0&0&0&0\\ \hline
0&0&0&1&0&0&0&0\\ \hline
0&0&0&1&0&0&0&0\\  \hline
0&0&0&0&0&1&0&0\\ \hline
0&0&0&0&1&1&0&0\\ \hline
0&0&0&0&0&1&1&\\ \hline
0&1&0&0&0&0&0&0\\ \hline
1&1&0&0&0&0&0&0
\end{array}\right]
\rightarrow
\left[\begin{array}{c|c|c|c|c|c|c|c}
0&0&1&1&0&0&0&0\\ \hline
0&0&0&1&0&0&0&0\\ \hline
0&0&0&1&0&0&0&0\\  \hline
0&0&0&0&0&1&0&0\\ \hline
0&0&0&0&1&1&0&0\\ \hline
0&0&0&0&0&1&1&1\\ \hline
0&1&0&0&0&0&0&0\\ \hline
1&1&0&0&0&0&0&0
\end{array}\right]
\]
}\endproof \end{example}

 Concerning Theorem \ref{lem:alg-works} the requirement that the row and column sums are known is essential. For instance, many different convex matrices have an  empty essential set, e.g., the two matrices where the first has its ones in the first row, and the second has its ones in the first column. However, then the row and column sums do not coincide. Also, the algorithm can possibly be made  more efficient in that
as soon as one places a 1, then because of convexity, that row to the east, and that column to the south can be completed knowing $R$ and $S$.

\section{Subclasses  and characterizations }
\label{sec:subclass-character}

 The well-known Gale-Ryser theorem (see e.g. \cite{BR91}) says that there is a $(0,1)$-matrix $A$ with given row sum vector $R$ and column sum vector $S$ if and only if $S \preceq R^*$. Here $\preceq$ denotes majorization order and $R^*$ is the conjugate of $R$. If one also requires that the matrix $A$ is convex, the situation is much more complicated. In fact (see a previous remark), it is {\em NP}-complete to decide for given $R$ and $S$ if there is a convex $(0,1)$-matrix with $R$ and $S$ as row and column sum vectors. The same is true for polyominos (obtained by adding the requirement of connectedness); see \cite{W} for more on these complexity questions. With this background it is natural to ask  if there are subclasses of $\mc{C}_{m,n}(R,S)$ where it is possible to characterize when the class is nonempty. We study this question in this section. 

First we consider the special case  $\mc{C}_{m,n}(R,S)$, where
$R=e$, the all 1's vector of size $m$, and $S=(s_1,s_2,\ldots,s_n)$ is   a nonnegative, integral vector.
\begin{lemma}
\label{lem:convex-e}
 The set $\mc{C}_{m,n}(e,S)$ is nonempty if and only if $S=(s_1,s_2,\ldots,s_n)$ satisfies $\sum_{j=1}^n s_j=m$. 
\end{lemma}
\begin{proof}
If $\mc{C}_{m,n}(e,S)$ is nonempty, then $\sum_j s_j= \sum_i r_i=m$. Conversely, if this condition holds, construct the $(0,1)$-matrix  $A=[a_{ij}]$ by $a_{11}= \cdots = a_{s_1,1}=1$, $a_{s_1+1,2}= \cdots = a_{s_1+s_2,2}=1$ etc, while the other entries are zero. Then $A \in \mc{C}_{m,n}(e,S)$.
\end{proof}

Therefore, when $R=e$, convexity places no further restriction on the possible column sum vectors. Similar conclusions hold when $S$ is the all ones vector. 

Next, we consider $R=ke=(k, k, \ldots, k)$ for some natural number $k\ge 2$. Here the situation is different, as the next example shows.


\begin{example}\label{ex:rowconstant}{\rm 
 Let $R=(2,2,2)$ and $S=(3,2,1)$. Then ${\mathcal C}(R,S)=\emptyset$, but with the permutation of $S$ given by $(1,3,2)$ we have the following convex matrix
\[\left[\begin{array}{c|c|c}
1&1& \\ \hline
&1&1 \\ \hline
&1&1 \end{array}\right].\]
If  $R=(3,3,3,3,3)$ and $S=(4,4,3,3,1)$,
 such a permutation does not exist in this case, since there are three 1's in row 1 and three 1's in row 5, and so by convexity we must have a column of sum 5.}\endproof\end{example}

 Let again $R=ke=(k, k, \ldots, k)$ and let $A$ be a $(0,1)$-matrix with row sum $R$. Then permuting rows of $A$ does not have any effect on row or column sums, but it may affect convexity as Example \ref{ex:rowconstant} shows. We do have the following lemma.

\begin{lemma}
 \label{lem:2-row-conv}
 Let $k$ be a positive integer,  and let $S$ be a nonnegative integral vector of length $n$.  Then $\mc{C}_{m,n}(ke,S)$ is nonempty if and only if there is a row convex matrix in $\mc{A}_{m,n}(ke,S)$. 
\end{lemma}
\begin{proof}
 Since a convex matrix is also row convex, we only need to prove that if $\mc{A}_{m,n}(R,S)$ contains a 
 row convex matrix $A$, then there is also a convex matrix in that class. Let $l_i$  denote the position of the left-most 1 in row $i$ of $A$ ($i \le m$). Reorder rows in $A$ according to increasing value of  $l_i$ ($i \le m$), and let $A'$ be the resulting matrix. Clearly $A'$ is row-convex, and it has a staircase pattern, with $k$ ones in every row. From this one can easily verify that it is also column convex. Thus $A' \in \mc{C}_{m,n}(R,S)$, as desired.
\end{proof}


Next, we consider the special case $k=2$. 
The following result characterizes the possible column sums of convex $m \times n$ $(0,1)$-matrices with two ones in every row.

\begin{theorem}
\label{thm:2-S-char}
Let $S=(s_1,s_2, \ldots, s_n)$ be a nonnegative integral vector. Then $\mc{C}_{m,n}(2e,S)$ is nonempty if and only if $S$ satisfies $\sum_{j=1}^n s_j=2m$ and 
\begin{equation}
 \label{eq:S-char}
   \sum_{i=1}^j (-1)^{j-i} s_i \;\;\;(j=1, 2, \ldots, n-1).
\end{equation}
\end{theorem}
\begin{proof}
Assume that the class $\mc{C}_{m,n}(2e,S)$ is nonempty. Then $\sum_j s_j=\sum_i r_i=2m$. Moreover, $\mc{C}_{m,n}(2e,S)$ contains a matrix $A$ with staircase pattern  so the position of the leftmost 1 is weakly increasing by rows.  Consequently there are nonnegative  integers $k_1, k_2, \ldots, k_{n-1}$ such that 
the $k_1$ first rows has initial 1 in column 1, the next $k_2$ rows  has initial 1 in column 2, etc.  Since each row has two consecutive ones the column sums are given by 
\begin{equation}
\label{eq:s-1}
    s_j=k_{j-1}+k_j \;\;\;(j=1, 2, \ldots, n)
\end{equation}
where $k_0=k_n=0$. Therefore $k_j=s_j-k_{j-1}$ for each $j$ and 
\[
  k_1=s_1, \; k_2=s_2-s_1, \; k_3=s_3-k_2=s_3-s_2+s_1, 
\]
and in general $k_j$ is an alternating sign sum of $s_1, s_2, \ldots, s_j$
\begin{equation}
 \label{eq:s-2}
   k_j=\sum_{i=1}^j (-1)^{j-i} s_i \;\;\;(j=1, 2, \ldots, n-1).
\end{equation}
Since all these numbers are nonnegative, we obtain (\ref{eq:S-char}). Conversely, assume  $\sum_{j=1}^n s_j=2m$ and (\ref{eq:S-char}) hold, and define $k_j$ ($j \le n-1$) by (\ref{eq:s-2}). Then $k_j\ge 0$ ($j \le n-1$) and, also,  (\ref{eq:s-1}) holds. So, if we construct a staircase matrix $A$ based on the $k_j$'s as above, this matrix will be in $\mc{C}_{m,n}(2e,S)$ as $\sum_{j=1}^n s_j=2m$, and the proof is complete.
\end{proof}

\begin{example}
{\rm 
Let $k=2$ and $S=(3,2,1)$. In Example \ref{ex:rowconstant} we saw that $\mc{C}_{3,3}(2e,S)$ is empty. In fact, the inequalities in (\ref{eq:S-char}) are 
\[
   s_1 \ge 0, \;\; s_2-s_1 \ge 0, \;\; s_3-s_2+s_1 \ge 0
\]
and the second inequality is violated: $2-3=-1 \not \ge 0$. If, however, $S=(1,3,2)$, then all these inequalities hold, and the matrix 
\[
\left[
\begin{array}{c|c|c}
1&1& \\ \hline
&1&1 \\ \hline
&1&1 
\end{array}
\right]
\]
lies in $\mc{C}_{3,3}(2e,S)$. Here $k_1=1$, $k_2=2$ and $k_3=0$.
} \endproof
\end{example}

The next result deals with convex matrices with both equal row sums and equal column sums. 

\begin{theorem}
\label{pr:convex-k-e}
 Let $k \le m$, $l \le n$ be positive integers. 
 The class $\mc{C}_{m,n}(ke,le)$ is nonempty if and only if  $m=pk$ and $n=pl$ for some integer $p$.
 \end{theorem}
\begin{proof}
Assume that $m=pk$ and $n=pl$ for an integer $p$. 
Then the matrix (direct sum) 
\[
   A= J_{k,l} \oplus J_{k,l} \oplus \cdots \oplus J_{k,l}
\]
where $J_{k,l}$ occurs $p$ times, 
lies in $\mc{C}_{m,n}(ke,le)$. Next, assume that $A=[a_{ij}]$ lies in $\mc{C}_{m,n}(ke,le)$. Consider $i$ minimal such that $a_{i1}=1$.  By convexity, 
\[
    a_{i1}=a_{i+1,1}= \cdots = a_{i+k-1,1}=1.
\]
Therefore, again by convexity, the $k \times l$ submatrix of $A$ containing rows $i, i+1, \ldots, i+k-1$ and columns $1, 2, \ldots, l$ equals $J_{k,l}$. The remaining entries in these $k$ rows, and in these $l$ columns, are zero. So, deleting these $k$ rows and $l$ columns, we can repeat this argument. After identifying 
$\min\{\lfloor m/k \rfloor, \lfloor n/l \rfloor\}$ such submatrices $J_k$ (that do not have any lines in common), we are left with a submatrix with either fewer rows than $k$ or fewer columns than $l$. The only possibility is that there are no rows and no columns left, otherwise it would contradict that  $A \in \mc{C}_{m,n}(ke,le)$. We conclude that $m/k=n/l$ must be an integer, as desired.
\end{proof}

Moreover, from the proof above, we see that when $m=pk$, $n=pl$ 
\[
   \mc{C}_{m,n}(ke,le)=\{P \otimes J_{k,l}: \mbox{$P$ is a permutation matrix of order $p$}\}
\]
where $\otimes$ denotes Kronecker product. Thus, $|\mc{C}_{m,n}(ke,le)|=p!$.

We now consider a special, but natural, class of convex matrices. We define a {\it Ferrers matrix} to be a $(0,1)$-matrix the positions of whose 1's form a Ferrers array. 
 Let $m_1, m_2, n_1,n_2$ be positive integers and consider a matrix 
\begin{equation}
 \label{eq:Ferrer-convex}
 A=
 \left[
\begin{array}{c|c}
 A_{11} &A_{12} \\ \hline
A_{21} &A_{22} \\ 
\end{array}
\right]
\end{equation}
where $A_{22}$, $A_{12}$, $A_{21}$, $A_{11}$ are matrices obtained by rotation of some (possibly different) Ferrers matrices $0$, $90$, $180$ and $270$ degrees counter clockwise, respectively, where $A_{11}$ has size $m_1 \times n_1$ and $A_{22}$ has size $m_2 \times n_2$, thereby determining the sizes of the other two matrices. We call $A$ a {\em Ferrers-convex matrix}. 
Let $R(A_{ij})$ be the row sum vector and $R'(A_{ij})$ be the column sum vector of $A_{ij}$ ($i,j \le 2$). Then $R'(A_{ij})$ is the conjugate of $R(A_{ij})$, but with the components in the reverse order for $A_{11}$ and $A_{21}$. Then clearly the row sum  and column sum vectors of $A$ are given by 
\[
\begin{array}{ll}\vspace{0.1cm}
   R(A)=(R(A_{11})+R(A_{12}), R(A_{21})+R(A_{22})),  \\ 
   S(A)=(R'(A_{11})+R'(A_{21}), R'(A_{21})+R'(A_{22})). 
\end{array}   
\]
\begin{example}{\rm \label{ex:Ferrers}
An example of a Ferrers-convex matrix of size $9 \times 7$ is 
\[
A=\left[
\begin{array}{cccc|ccccc}
&&&1&&&&&\\ 
&&1&1&1&&&&\\ 
&&1&1&1&1&&&\\ 
1&1&1&1&1&1&1&\\ \hline 
&1&1&1&1&1&1&1&1\\ 
&&&1&1&&&\\ 
&&&1&&&&&\\ 
\end{array}
\right].
\]
}\endproof\end{example}

\begin{lemma}
\label{lem:Ferrers-convex}
  Every Ferrers-convex matrix is convex and has unimodal row and column sum vectors. 
 \end{lemma}
\begin{proof}
 Consider $A$ as in (\ref{eq:Ferrer-convex}). Then the ones in each row are consecutive since the ones in $A_{11}$ and $A_{21}$ are right-justified, and the ones in $A_{12}$ and $A_{22}$ are left-justified. Similarly, the ones are consecutively in each column. Thus, $A$ is convex. Moreover, from the Ferrers property, the row sums are nondecreasing in the first $m_1$ rows, and nonincreasing in the last $m_2$ rows, so the row sum vector is unimodal. For similar reasons the column sum vector is unimodal.
\end{proof}

Thus, Ferrers-convex matrices form a large class of convex matrices obtained  from all possible partitions of four integers being the number of ones in each of the four submatrices in (\ref{eq:Ferrer-convex}). 

\begin{example}{\rm \label{ex:Ferrers2}
Let $A_n=[a_{ij}]$ be the $(0,1)$-matrix of order $n$ with 1's on the diagonal and superdiagonal, i.e., $a_{ij}=1$ when $i\le j \le i+1$ ($i <n$) and $a_{nn}=1$. For instance, 

\[
A_4=
\left[
\begin{array}{c|c|c|c}
1&1&0&0\\ \hline
0&1&1&0\\ \hline
0&0&1&1\\ \hline
0&0&0&1
\end{array}
\right].
\]
Then, for each $n$,  $A_n$ is connected, convex and unimodal, as $R(A_n)=(2, 2, \ldots, 2, 1)$ and $S(A_n)=(1, 2, 2, \ldots, 2)$. However, one can easily check that $A_n$ is not Ferrers-convex. Thus, a converse of Lemma \ref{lem:Ferrers-convex} does not hold.

Even Ferrers-convex matrices are not uniquely determined by their row and column sum, as Example 
\[\left[\begin{array}{cc}
1&0\\ 1&1\\ 0&1\end{array}\right]\mbox{ and }
\left[\begin{array}{cc}
0&1\\ 1&1\\ 1&0\end{array}\right].\]
with $R=(1,2,1)$  and  $S=(2,2)$ shows.
}\endproof\end{example}

\section{Interchanges and convex-classes}
\label{sec:int-convex}

We continue to investigate convex matrices, but now with a focus on interchanges and preservation of convexity.

Let $J_{k,l}$ as usual be the $k \times l$ all ones matrix. Let $J^1_{k,l}$ be obtained from $J_{k,l}$ by replacing each of the ones in positions $(1,1)$ and $(k,l)$ by a zero. Similarly, let $J^2_{k,l}$ be obtained from $J_{k,l}$ by replacing each of the ones in positions $(1,l)$ and $(k,1)$ by a zero. Note that an interchange in $J^1_{k,l}$ involving the first and last row and column gives $J^2_{k,l}$.

\begin{proposition}
\label{pr:convex-interchange}
 Let  $A \in \mc{C}_{m,n}$ be such that each row and column sum is at least $2$. Then  $A$ permits an interchange into another matrix $B$ in $\mc{C}_{m,n}$ if and only if $A$ has the form 
 \begin{equation}
  \label{eq:A-form}
  A= 
  \left[
  \begin{array}{ccc}
    A_{11} & A_{12} & A_{13} \\
    A_{21} & A_{22} & A_{23} \\
    A_{11} & A_{32} & A_{33} \\
  \end{array}
  \right]
  \end{equation}
  where $(i)$ $A_{22}$ equals either $J^1_{k,l}$ or $J^2_{k,l}$ for some $k,l$, and $(ii)$ the first and last row of $A_{21}$ and $A_{23}$ are zero, and $(iii)$ the first and last column of $A_{12}$ and $A_{32}$ are zero. 
 \end{proposition}
\begin{proof}
Let $A$ be as in (\ref{eq:A-form}) with the block matrices as stated in the theorem where $A_{22}=J^1_{k,l}$.  Then we may apply an interchange in $A$  involving the first and last row and column of $A_{22}$. The resulting matrix $B$ is as $A$ except that the block $A_{22}$ is replaced by $J^2_{k,l}$. The zeros described in conditions (ii) and (iii) assure that $B$ is a convex matrix.

We shall prove the  converse. In a convex matrix the ones in every line (row or column) are consecutive, and those positions form an {\em interval}. The following observation follows from the assumption that each line sum is at least 2. 
\begin{itemize}
 \item If an interchange in a convex matrix results in another convex matrix, then in each of the four modified lines the interval is changed by shifting the  interval one position. 
\end{itemize}

Assume now that $A$ permits an interchange into another matrix $B$ in $\mc{C}_{m,n}$. We can partition $A$ as in (\ref{eq:A-form}) such that the interchange involves the first and last rows and columns of the submatrix $A_{22}$. By symmetry we may assume that $A_{22}$ has zeros on the upper right and the lower left corner positions. Then, by the observation above, all entries in the first and last row and column of $A_{22}$ are 1, except the two mentioned zeros. Convexity then implies that the remaining entries in $A_{22}$ are 1, so $A_{22}=J^1_{k,l}$. This proves the converse.
\end{proof}

Let $R=(r_1,r_2,\ldots,r_m)$ and $S=(s_1,s_2,\ldots,s_n)$.
The class ${\mathcal A}(R,S)$ is a {\it convex-class} provided every matrix in ${\mathcal A}(R,S)$ is convex, that is, ${\mathcal A}(R,S)={\mathcal C}(R,S)$.
 Consider a convex-class $\mc{A}(R,S)$ and let $A=[a_{ij}] \in \mc{A}(R,S)$. 
 Let $i \le m$ and consider $I=\{j: a_{ij}=1\}$ which we call  a  {\em horizontal interval}. Similarly, we have the {\em vertical interval } $\{i: a_{ij}=1\}$ for each $j$.  If $I,I'$ are two horizontal intervals of the form  $I=\{k, k+1, \ldots, l\}$ and $I'=\{k+1, k+2, \ldots, l+1\}$ for some $k<l$ (so, in particular,  $|I|=|I'|$), we say that $I'$ is a {\em $1$-shift} of $I$, and vice versa. We also say that $I, I'$ is a 1-shift {\em pair}. A similar notion applies to vertical intervals. All intervals are nonempty as there are no zero lines. The next result establishes very strong requirements on the structure of matrices in a convex-class, requiring $m\choose 2$ pairwise comparisons of rows and $n\choose 2$ pairwise comparisons of columns.

\begin{theorem}
\label{thm:convex-class}
 Assume that  $\mc{A}(R,S)$ is a {\em convex-class}, and let $A \in \mc{A}(R,S)$. Let $I$ and $I'$ be  horizontal intervals  associated with different rows of $A$. Then $($at least$)$ one of the following holds:
 
 $(i)$ $|I|=|I'|=1$;
 
 $(ii)$ $I \subseteq I'$ or $I' \subseteq I$;
 
 $(iii)$ $I'$ is a $1$-shift of $I$.

 \noindent
 In particular, if for some $k\ge 1$, $|I|=k$ and $|I'|\ge k+1$, then $I\subseteq I'$.
 
 Similar properties hold for the vertical intervals. Moreover, every interchange in $\mc{A}(R,S)$ either takes place in two rows $($or columns$)$ with a single $1$, or two rows $($or columns$)$ corresponding to a $1$-shift pair. 
 \end{theorem}
\begin{proof} First note that the conditions (i), (ii), and (iii) imply that if $|I|=|I'|=1$ and $I\ne I'$, then $I=\{k\}$ and $I'=\{k+1\}$ for some $k$; also,  if two horizontal intervals or two vertical intervals are disjoint, then they both contain exactly one element. 

 We assume that (i) and (ii) do not hold and prove (iii) holds. Let $I=\{i, i+1, \ldots, k\}$ and $I'=\{i', i'+1, \ldots, k'\}$ where we may assume $i<i'$, since by symmetry, we may take $i \le i'$, and if $i=i'$, (ii) would be violated. Then  $k<k'$, due to (ii). Moreover, $A$ contains ones in the positions corresponding to $i$ in $I$ (that is, the row associated with $I$ and column $i$) and to $k'$ in $I'$ (the row associated with $I'$ and column $k'$). 
 Therefore  we can make an interchange in $A$ involving these two ones (as the other two positions involved contain zeros). Let $A'$ be the resulting matrix.  However, $A'$ has consecutive ones in these two rows if and only if $i'=i+1$ and $k'=k+1$, i.e., when $I'$ is a $1$-shift of $I$. Here we used that  at least one of the intervals $I$ and $I'$ contains more than one element.  Thus (iii) holds. Clearly, a similar arguments applies to column intervals. 
 
 The last statement of the theorem follows from the first part as the two rows of the interchange cannot be of type (ii). 
\end{proof}

\begin{example}{\rm \label{ex:nonconvex} Let $R=(2,3,3)$ and $S=(2,3,2,1)$. The matrix 
\[\left[\begin{array}{c|c|c|c}
1&1&&\\ \hline
1&1&1&\\ \hline
&1&1&1\end{array}\right]\]
is a convex matrix in  ${\mathcal A}(R,S)$,  but rows 1 and 3, and columns 1 and 4 do not satisfy the conditions in Theorem \ref{thm:convex-class}, and indeed an interchange replaces $A$ with a non-convex matrix.

Now suppose that $R=(2,1,2)$ and $S=(3,1)$. Then the matrix in ${\mc A}(R,S)$
\[\left[\begin{array}{cc}
1&1\\
1&0\\
1&1\end{array}\right]\]
satisfies (i), (ii), and (iii) of Theorem \ref{thm:convex-class}  but is not convex, and so ${\mc A}(R,S)$ is not a convex class.
}\hfill{$\Box$}\end{example}

 Concerning the structure of convex-classes an interesting special case is when $r_i =n$ for some $i$; we then say that row $i$ is {\em full}. Let $A \in \mc{A}(R,S)$.  Then convexity implies that the full rows are consecutive, say in rows $p, p+1, \ldots, p'$ and the remaining intervals above and below are nested:
 \[
    I_1 \subseteq I_2 \subseteq \cdots \subseteq I_{p} = \{1, 2, \ldots, n\}=\cdots = I_{p'} \supseteq
         I_{p'+1} \supseteq \cdots \supseteq I_m
 \]
 where $I_i$ is the interval associated with row $i$ ($i \le m$). 
   The previous theorem may be used to construct convex-classes, as the next example illustrates.

\begin{example}
{\rm 
Let $R=(1,4,6,4,1)$ and $S=(1,2,4,4,3,2)$, and the following convex matrix in $\mc{A}(R,S)$
 \[
 A=\left[
 \begin{array}{c|c|c|c|c|c}
0&0&{\bf 1}&0&0&0\\ \hline
0&{\bf 1}&1&1&1&0\\ \hline
1&1&1&1&1&1\\  \hline 
0&0&1&1&1&{\bf 1}\\ \hline
0&0&0&{\bf 1}&0&0
\end{array}
\right].
\]
The interval in row 4 is a 1-shift of the interval in row 2. An interchange in these two rows and columns 2 and 5 gives the convex matrix 
\[
 B=\left[
 \begin{array}{c|c|c|c|c|c}
0&0&1&0&0&0\\ \hline
0&0&1&1&1&1\\ \hline
1&1&1&1&1&1\\  \hline 
0&1&1&1&1&0\\ \hline
0&0&0&1&0&0
\end{array}
\right].
\]
The first and last rows each have only one 1 and they permit an interchange using columns 3 and 4. The ones involved in these interchanges are indicated in boldface. All these 4 matrices constructed from these interchanges are convex, and there are no other matrices in $\mc{A}(R,S)$, so this is a convex class. \endproof
}
\end{example}

We now generalize this example and construct a family of convex-classes. This is done by selecting intervals (consecutive integers) with a certain property and then embedding these as the row intervals in a matrix. 

\bigskip
\centerline{\bf Algorithm to construct a convex-class $\mc{A}(R,S)$}

\begin{itemize}
\item[1.] Choose a 1-shift pair $I_1, I'_1$.

\item[2.]  for $i=2, 3, \ldots, k$ \\
Choose a 1-shift pair $I_i, I'_i$ such that $I_i \cap I'_i \supseteq I_{i-1} \cup I'_{i-1}$.

\item[3.] Construct the matrix $A$ with $2k$ rows corresponding to these intervals: $I_1$ and $I'_1$ are the first and last row, $I_2$ and $I'_2$ are the second and second last row etc. Do this in such a way that $A$ has no zero columns. 
\end{itemize}

We obtain the following result. 

\begin{corollary}
\label{cor:convex-class-example}
The algorithm constructs a convex-class $\mc{A}(R,S)$ consisting of  $2^k$ matrices.
\end{corollary}
\begin{proof}
First, the constructed matrix is convex: it is clearly row-convex, and the condition $I_i \cap I'_i \supseteq I_{i-1} \cup I'_{i-1}$ assures that it is also column-convex. 
Next, no interchange is possible for two rows when one interval is contained in the other. Therefore, the only possible interchanges are for each of the pairs $I_i, I'_i$ ($i\le k$). Each of these gives a new convex matrix, due to mentioned intersection condition,  and these interchanges are ``independent", so that the class consists of $2^k$ matrices.
\end{proof}

It would be of interest to characterize the  pairs  $R,S$ for which there is a unique convex matrix in 
${\mathcal A}(R,S)$, that is, for which  $|{\mathcal C}(R,S)|=1$.

\section{A Partially Ordered Set}
\label{sec:poset}

Consider the partially ordered set  $({\mc C}_{m,n},\le)$ of all $m\times n$ convex matrices with the entrywise partial 
order. Starting with a  connected matrix $C\in {\mc C}_{m,n}$, we can sequentially  change 0's to 1's, one 0 at a time,  in such a way that  we obtain $J_{m,n}$  with all intermediary  matrices  convex.
 If $C$ is not connected, then by applying this procedure to each connected component of  $C$, we can first arrive at a convex matrix of the form $P(J_{m_1,n_1},\ldots,J_{m_p,n_p})$ for some $p\times p$ permutation matrix $P$, and then sequentially  change 0's to 1's with all resulting matrices convex and arrive at  $J_{m,n}$. In a similar way, starting with any convex matrix $A\in {\mc C}_{m,n}$, we can sequentially change 1's to 0's and arrive at the zero matrix $O_{m,n}$ with all intermediary matrices convex.

It follows that maximal chains in $({\mc C}_{m,n},\le)$ have length $ mn$ (i.e., contain $(mn+1)$ matrices) 
and every $C\in {\mc C}_{m,n}$ belongs to a maximal  chain. Also
if $C_1$ and $C_2$ are in ${\mc C}_{m,n}$, then $C_2$ {\it covers } $C_1$ provided $C_1\le C_2$ and the number of 1's in $C_2$ is one more than the number of 1's in $C_1$.
If $C_1,C_2\in {\mc C}_{m,n}$, then the {\it intersection} $C_1\wedge C_2$ of $C_1$ and $C_2$ is the $m\times n$ $(0,1)$-matrix $C_1\wedge C_2$ which has 1's in exactly those places where both $C_1$ and $C_2$ have 1's.

\begin{lemma}\label{lem:wedge}
If $C_1,C_2\in {\mc C}_{m,n}$, then $C_1\wedge C_2\in {\mc C}_{m,n}$ and $C_1\wedge C_2$ is the meet $($i.e., GLB$)$  of $C_1$ and $C_2$ in ${\mc C}_{m,n}$.
\end{lemma}

\begin{proof} Suppose, e.g., there are two 1's in a row of $C_1\wedge C_2$. Then in between these 1's there are only 1's in $C_1$ and $C_2$. It follows that $C_1\wedge C_2$ is a convex matrix and is the meet of $C_1$ and $C_2$.
\end{proof}
\begin{theorem}\label{th:lattice}
The partially ordered set $({\mc C}_{m,n},\le)$ is a lattice.
\end{theorem}

\begin{proof}
Since ${\mc C}_{m,n}$ is a finite set with a meet, then every two convex matrices $C_1$ and $C_2$ in ${\mc C}_{m,n}$ have a {\it join} $C_1\vee C_2$ (i.e., LUB), namely the meet of all the matrices $X\in {\mc C}_{m,n}$ with $C_1,C_2\le X$.
\end{proof}

Let $A$ be an arbitrary $(0,1)$-matrix of size $m \times n$. We define the {\it convex hull}  $\overline{A}^c$ of $A$ to be the meet of all the convex sets $C$ in $({\mc C}_{m,n},\le)$ such that $A\le C$.  The convex hull $\overline{A}^c$ of $A$ is obtained by replacing every 0 between two 1's with a 1, repeatedly, until one obtains a convex matrix.

\begin{example}\label{ex:closure} {\rm
A matrix $A$ and its convex closure $\overline{A}^c$ are illustrated below:
\[A=\left[\begin{array}{c|c|c|c}
1&1&&\\ \hline
1&&&1\\ \hline
&&1&1\\ \hline
&&&\end{array}\right], \quad  \overline{A}^c=\left[\begin{array}{c|c|c|c}
1&1&&\\ \hline
1&1&1&1\\ \hline
&&1&1\\ \hline
&&&\end{array}\right].
\] \endproof

}
\end{example}

The essential set of a convex matrix was defined relative to the southeast (SE) and so it is more appropriately called the {\it SE-essential set}.
It could just as well have been defined relative to the other three possible directions: the   {\it SW-essential set}, {\it NE-essential set}, and  {\it NW-essential set}. For a convex matrix $C\in {\mc C}_{m,n}$, let ${\mc E}(A)$
be the union of these four  essential sets. A specific one of these essential sets is empty exactly when there is a 1 in the oppositely directed  corner. Thus ${\mc E}(C)=\emptyset$ if and only if there is a 1 in each of the four corners of $C$, and since $C$ is convex, if and only if $C=J_{m,n}$.

We have the following theorem.

\begin{theorem}\label{th:all}
Let $C\in {\mc C}(m,n)$.  Then  $C'\in {\mc C}_{m,n}$ covers $C$ in $({\mc C}_{m,n},\le)$ if and only if 
$C'$ is obtained from $C$ by replacing the $0$ in some position of ${\mc E}(C)$ with a $1$. In particular, the number of convex matrices in ${\mc C}_{m,n}$ that cover $C$ equals $|{\mc E}(C)|$.
\end{theorem}

\begin{proof} A position belongs to ${\mc E}(C)$ if and only if it is occupied by a 0, and there is a 1 opposite it in the two positions S and E, N and E, N and W, or S and W. These are exactly the positions corresponding to the four essential sets.
\end{proof}

\begin{example}{\rm \label{ex:bigess}
Consider the convex matrix
\[C=\left[\begin{array}{c|c|c|c|c|c}
&&&1&1&1\\ \hline
&&1&1&1&\\ \hline
&&1&1&&\\ \hline
&1&1&1&&\\ \hline
1&1&1&&&\\ \hline
&&1&&&\end{array}\right]\rightarrow
\left[\begin{array}{c|c|c|c|c|c}
&&a&1&1&1\\ \hline
&&1&1&1&c\\ \hline
&a&1&1&c&\\ \hline
a&1&1&1&&\\ \hline
1&1&1&c&&\\ \hline
&b&1&&&\end{array}\right],\]
where $a$ denotes a position in the SE-essential set, $b$ denotes a position in the NE-essential set, and $c$ denotes a position in the NW-essential set. The SW-essential set is empty.
Then it is easy to check that these positions in ${\mc E}(C)$ are  those  and only those positions in which a 1 can be inserted to get a convex set.
}\hfill{$\Box$}\end{example}

It is natural to ask whether $({\mc C}_{m,n},\le)$ is a distributive lattice, that is, whether
\[C_1\vee (C_2\wedge C_3)=(C_1\vee C_2)\wedge (C_1\vee C_3).\]
That this need not hold is already seen with the simple example
\[C_1=\left[\begin{array}{c|c|c}
1&0&0\end{array}\right], C_2=\left[\begin{array}{c|c|c}
0&1&0\end{array}\right], C_3=\left[\begin{array}{c|c|c}
0&0&1\end{array}\right], \]
where
\[C_1\vee (C_2\wedge C_3)=\left[\begin{array}{c|c|c}
1&0&0\end{array}\right]\mbox{ and } (C_1\vee C_2)\wedge (C_1\vee C_3)=\left[\begin{array}{c|c|c}
1&1&0\end{array}\right].\]

\end{document}